\documentclass[centertags,12pt]{amsart}
\usepackage{latexsym}
\usepackage{amsthm}
\usepackage{amssymb}
\usepackage{color}
\usepackage{tikz}
\usepackage{wasysym}
\usepackage[foot]{amsaddr}

\usepackage[utf8]{inputenc}
\usepackage{mathrsfs}
\usepackage[english]{babel}
\usepackage{setspace}
\usepackage{ textcomp }
\usepackage{graphicx}
\usepackage{cite}
\usepackage{comment,cancel}
\graphicspath{ {images/} }

\numberwithin{equation}{section}
\makeatletter
\renewcommand{\subsection}{\@startsection
{subsection}{2}{0mm}{\baselineskip}{-0.25cm}
{\normalfont\normalsize\bf}}
\makeatother

\newtheorem{theorem}{Theorem}[section]
\newtheorem{proposition}[theorem]{Proposition}
\newtheorem{lemma}[theorem]{Lemma}
\newtheorem{corollary}[theorem]{Corollary}

   {\theoremstyle{definition}

}
   \theoremstyle{remark}

\newcommand{\F}{{\mathbb F}}

\newcommand{\PG}{{\rm PG}}

\begin{document}

\author[]{}

\author[D. Bartoli]{Daniele Bartoli}
\address{
Dipartimento di Matematica e Informatica, Universit\`a degli Studi di Perugia}
	\email{daniele.bartoli@unipg.it}

\author[M. Montanucci]{Maria Montanucci}
\address{Department of Applied Mathematics and Computer Science, Technical University of Denmark, Asmussens All\'e, building 303B, DK‐2800 Kongens Lyngby, Denmark}
	\email{marimo@dtu.dk}

\title{Towards the full classification of exceptional scattered polynomials}

\maketitle


\begin{abstract}
Let $f(X) \in \mathbb{F}_{q^r}[X]$ be a $q$-polynomial. If the $\mathbb{F}_q$-subspace $U=\{(x^{q^t},f(x)) \mid x \in \mathbb{F}_{q^n}\}$ defines a maximum scattered linear set, then we call $f(X)$ a scattered polynomial of index $t$.  The asymptotic behaviour of scattered polynomials of index $t$ is an interesting open problem. In this sense, exceptional scattered polynomials of index $t$ are those for which $U$ is a maximum scattered linear set in ${\rm PG}(1,q^{mr})$ for infinitely many $m$. The complete classifications of exceptional scattered monic polynomials of index $0$ (for $q>5$) and of index 1 were obtained in \cite{BZ}. In this paper we complete the classifications of exceptional scattered monic polynomials of index $0$ for $q \leq 4$. Also, some partial classifications are obtained for arbitrary $t$. As a consequence, the complete classification of exceptional scattered monic polynomials of index $2$ is given.
\end{abstract}

\vspace{0.5cm}\noindent {\em Keywords}:
maximum scattered linear set; MRD code; algebraic curve; Hasse-Weil bound.
\vspace{0.2cm}\noindent

\section{Introduction}
Let $q$ be a prime power and $r,n\in \mathbb{N}$. Let $V$ be a vector space of dimension $r$ over $\mathbb{F}_{q^n}$. For any $k$-dimensional $\mathbb{F}_q$-vector subspace $U$ of $V$, the set $L(U)$ defined by the nonzero vectors of $U$ is called an $\mathbb{F}_q$-\emph{linear set} of $\Lambda=\PG(V, q^n)$ of \emph{rank} $k$, i.e.
\[ L(U)=\{\langle {\mathbf u} \rangle_{\mathbb{F}_{q^n}}: {\mathbf u} \in U\setminus \{\mathbf{0} \}  \}.\]
It is notable that the same linear set can be defined by different vector subspaces. Consequently, we always consider a linear set and the vector subspace defining it in pair. 

Let $\Omega=\PG(W,\mathbb{F}_{q^n})$ be a subspace of $\Lambda$ and $L(U)$ an $\mathbb{F}_q$-linear set of $\Lambda$. We say that $\Omega$ has \emph{weight} $i$ in $L(U)$ if $\dim_{\mathbb{F}_q}(W\cap U)=i$. Thus a point of $\Lambda$ belongs to $L(U)$ if and only if it has weight at least $1$. Moreover, for any $\mathbb{F}_q$-linear set $L(U)$ of rank $k$, 
\[|L(U)|\leq \frac{q^{k}-1}{q-1}.\]
When the equality holds, i.e.\ all the points of $L(U)$ have weight $1$, we call $L(U)$ a \emph{scattered} linear set. A scattered $\mathbb{F}_q$-linear set of highest possible rank is called a \emph{maximum scattered} $\mathbb{F}_q$-\emph{linear set}. See \cite{blokhuis_scattered_2000} for the possible ranks of maximum scattered linear sets.

Maximum scattered linear sets have various applications in Galois geometry, including blocking sets \cite{ball_linear_2000,lunardon_linear_k-blocking_2001,lunardon_blocking_2000}, two-intersection sets \cite{blokhuis_scattered_2000,blokhuis_two-intersection_2002}, finite semifields \cite{cardinali_semifield_2006,ebert_infinite_2009,lunardon_maximum_scattered_2014,marino_towards_2011}, translation caps \cite{bartoli_maximum_2017}, translation hyperovals \cite{durante_hyperovals_2017}, etc. For more applications and related topics, see \cite{polverino_linear_2010} and the references therein. For recent surveys on linear sets and particularly on the theory of scattered spaces, see \cite{lavrauw_scattered_2016,lavrauw_field_reduction_2015}.

In this paper, we are interested in maximum scattered linear sets in $\PG(1, q^n)$. Let $f$ be an $\mathbb{F}_q$-linear function over $\mathbb{F}_{q^n}$ and
\begin{equation}\label{eq:U_(x,f)}
	U= \{ (x, f(x)) : x\in \mathbb{F}_{q^n}  \}.
\end{equation}
Clearly $U$ is an $n$-dimensional $\mathbb{F}_q$-subspace of $\mathbb{F}_{q^n}$ and $f$ can be written as a $q$-polynomial $f(X)=\sum a_i X^{q^i} \in \mathbb{F}_{q^n}[X]$. It is not difficult to show that a necessary and sufficient condition for $L(U)$ to define a maximum scattered linear set in $\PG(1, q^n)$ is
\begin{equation}\label{eq:scattered_poly}
	\frac{f(x)}{x} = \frac{f(y)}{y} ~\text{ if and only if }~ \frac{y}{x}\in \mathbb{F}_q, \quad\text{ for }x,y\in \mathbb{F}_{q^n}^*.
\end{equation}
In \cite{sheekey_new_2016}, such a $q$-polynomial is called a \emph{scattered polynomial}.

Two linear sets $L(U)$ and $L(U')$ in $\PG(2,q^n)$ are \emph{equivalent} if there exists an element of $\mathrm{P\Gamma L}(2, q^n)$ mapping $L(U)$ to $L(U')$. It is obvious that if $U$ and $U'$ are equivalent as $\mathbb{F}_{q^n}$-spaces, then $L(U)$ and $L(U')$ are equivalent. However, the converse is not true in general. For recent results on the equivalence issue and the classification of linear sets, we refer to \cite{csajbok_classes_2017,csajbok_maximum_2017,csajbok_equivalence_2016}.

There is a very interesting link between maximum scattered linear sets and the so called maximum rank distance (MRD for short) codes \cite{csajbok_maximum_2017}. In particular, a scattered polynomial over $\mathbb{F}_{q^n}$ defines an MRD code in $\mathbb{F}_q^{n\times n}$ of minimum distance $n-1$; see \cite{BZ} for more details.

Given a scattered polynomial $f$ over $\mathbb{F}_{q^n}$, an MRD code can be defined by the following set of $\mathbb{F}_q$-linear maps
\begin{equation}\label{eq:MRD_scattered}
	 C_f :=\{ax+bf(x) : a, b\in \mathbb{F}_{q^n}\}. 
\end{equation}
To show that \eqref{eq:MRD_scattered} defines an MRD code, we only have to prove that $ax+bf(x)$ has at most $q$ roots for each $a,b\in \mathbb{F}_{q^n}$ with $ab\neq 0$, which is equivalent to \eqref{eq:scattered_poly}. 

It is worth pointing out that the MRD code defined by \eqref{eq:MRD_scattered} is $\mathbb{F}_{q^n}$-linear. Using the terminology in \cite{lunardon_kernels_2016}, one of its nuclei is $\mathbb{F}_{q^n}$. The equivalence problem of $\mathbb{F}_{q^n}$-linear MRD codes is slightly easier to handle compared with other MRD codes; see \cite{morrison_equivalence_2014}. It can be easily proved that for two given scattered polynomials $f$ and $g$, if they define two equivalent MRD codes, then the two associated maximum scattered linear sets are also equivalent. However the converse statement is not true in general; see \cite{csajbok_classes_2017,sheekey_new_2016}.

To the best of our knowledge, up to the equivalence of the associated MRD codes, all constructions of scattered polynomials for arbitrary $n$ can be summarized as one family
\begin{equation}\label{eq:the_one}
	f(x)=\delta x^{q^s} + x^{q^{n-s}},
\end{equation}
where $s$ satisfies $\gcd(s,n)=1$ and $\mathrm{Norm}_{\mathbb{F}_{q^n}/\mathbb{F}_q}(\delta)=\delta^{(q^n-1)/(q-1)}\neq 1$. 

When $\delta=0$ and $n-s=1$, $f$ defines the maximum scattered $\mathbb{F}_q$-linear set in $\PG(1, q^n)$ found by Blokhuis and Lavrauw \cite{blokhuis_scattered_2000}. In fact, no matter which value $s$ takes, $f(x)=x^{q^s}$ defines the same maximum scattered $\mathbb{F}_q$-linear set. However, the MRD codes associated with $x^{q^s}$ and $x^{q^t}$ are inequivalent if and only if $s\not\equiv \pm t \pmod{n}$.

When $\delta\neq 0$, $f$ defines the MRD codes constructed by Sheekey in \cite{sheekey_new_2016} and the equivalence problem was completely solved in \cite{lunardon_generalized_2015}. In particular, when $s=1$, the associated maximum scattered $\mathbb{F}_q$-linear set in $\PG(1, q^n)$ was found  by Lunardon and Polverino \cite{lunardon_blocking_2001}. In \cite{csajbok_classes_2017}, it is claimed that for different $s$ the associated linear sets can be inequivalent.

Besides the family of scattered polynomials defined in \eqref{eq:the_one}, very recently, Csajb\'ok, Marino, Polverino and Zanella found another new family of MRD codes which are of the form 
\begin{equation}\label{eq:the_two}
f(x)=\delta x^{q^s} + x^{q^{n/2+s}},
\end{equation}
for $n=6, 8$ and some $\delta\in\mathbb{F}_{q^n}^*$; see \cite{csajbok_newMRD_2017}.

As scattered polynomials appear to be very rare, it is natural to look for some classifications of them. Given an integer $0\le t\le n-1$ and a $q$-polynomial $f$ whose coefficients are in $\mathbb{F}_{q^n}$, if 
\begin{equation}\label{eq:main_target}
	U_m= \{(x^{q^t}, f(x)): x\in \mathbb{F}_{q^{mn}}  \}
\end{equation}
defines a maximum scattered linear set in $\PG(1,q^{mn})$ for infinitely many $m$, then we call $f$ an \emph{exceptional scattered polynomial of index} $t$. {In particular, if $U_1$ is maximum scattered, then we say $f$ is a scattered polynomial over $\mathbb{F}_{q^n}$ of index $t$. } 

Note that  \eqref{eq:main_target}  is slightly different from  \eqref{eq:U_(x,f)}: in this ways we can describe  the unique known family \eqref{eq:the_one} as an exceptional one. Taking $t=s$, from \eqref{eq:the_one} we get
\[\{(x^{q^s},x + \delta x^{q^{2s}} ): x\in \mathbb{F}_{q^{mn}}   \}  \]
which defines a maximum scattered linear set for all $mn$ satisfying $\gcd(mn,s)=1$. This means $x + \delta x^{q^{2s}}$ is exceptional of index $s$.

Assume that $U_m$ given by \eqref{eq:main_target} defines a maximum scattered linear set for some $m$. Now, we want to normalize our research objects to exclude some obvious cases.

\begin{enumerate}
	\item[[ C1]] \label{Prima} Without loss of generality, we assume that the coefficient of $X^{q^t}$ in $f(X)$ is always $0$.
	\item[[ C2]]\label{Seconda} When $t>0$, we assume that the coefficient of $X$ in $f(X)$ is nonzero; otherwise let $t_0=\min\{i: a_i\neq 0\}$ and it is equivalent to consider
	\[ \left\{\left(x^{q^{t-t_0}}, \sum_{i=t_0}^{n-1}a_i^{q^{n-t_0}}x^{q^{i-t_0}}\right): x\in \mathbb{F}_{q^{mn}}  \right\} \]
	instead of $U_m$.
	\item[[ C3]]\label{Terza} We assume that $f(X)$ is monic.
\end{enumerate}

The main results in \cite{BZ} can be summirized as follows.
\begin{theorem}[\!\!\cite{BZ}] \label{t01}
\begin{enumerate}
	\item For $q>5$, $X^{q^k}$ is the unique exceptional scattered monic polynomial of index $0$.
	\item The only exceptional scattered monic polynomials $f$ of index $1$ over $\mathbb{F}_{q^n}$ are $X$ and $bX + X^{q^2}$ where $b\in \mathbb{F}_{q^n}$ satisfying $\mathrm{Norm}_{q^n/q}(b)\neq 1$. In particular, when $q=2$, $f(X)$ must be $X$.
\end{enumerate}
\end{theorem}

Scattered polynomials are related with algebraic curves via the following straightforward result; see also \cite[Lemma 2.1]{BZ}.

\begin{lemma}\label{le:link}
	The vector space $U= \{(x^{q^t}, f(x)): x\in \F_{q^{n}}\}$ defines a maximum scattered linear set $L(U)$ in $\PG(1, q^n)$ if and only if the curve defined by
	\begin{equation}\label{eq:curve_condition}
	\frac{f(X)Y^{q^t} - f(Y)X^{q^t}}{X^qY-XY^q}
	\end{equation}
	in $\PG(2,q^n)$ contains no affine point $(x,y)$ such that $\frac{y}{x}\notin\F_{q}$.
\end{lemma}

In this paper we  close the gaps left for $q\leq4$ in the above classification of index $0$ exceptional scattered polynomials, proving that Theorem \ref{t01} (1) holds also in these cases.
We also obtain partial results for exceptional scattered polynomials of index larger than $1$. More precisely, the following is the main result of the paper.
\begin{theorem} \label{main}
Let $t \geq 2$ be a natural number. Then the unique exceptional scattered polynomials of index $t$ are those having at least two non-trivial terms of $q$-degree less than $t$ (other than the one of $q$-degree zero).
\end{theorem}
Since for $t=2$ the condition required in Theorem \ref{main} is trivially satisfied by every $q$-polynomial $f(X)$, the complete classification of exceptional scattered polynomials of index $2$ is obtained.
\begin{corollary} \label{main1}
The only exceptional scattered monic polynomials $f$ of index $2$ over $\mathbb{F}_{q^r}$ are those of type \eqref{eq:the_one}. \end{corollary}

As in \cite{BZ}, the main idea consists in converting the original question  into an investigation of a special type of algebraic curves. Then approaches besed on intersection theory or function field theory together with the Hasse-Weil Theorem are used  to get contradictions.

\section{An approach based on intersection multiplicity}\label{Sec:Machinery}

In this paper we use investigations on singular points of curves $\mathcal{C}_f$ associated with scattered polynomials $f(X)$ (see Lemma \ref{le:link} above) to get information on the existence of absolutely irreducible components of $\mathcal{C}_f$ defined over $\mathbb{F}_q$.

This approach  has been used for the first time by Janwa, McGuire and Wilson \cite{janwa_double-error-correcting_1995} to classify functions on $\mathbb{F}_{p^n}$ that are almost perfect nonlinear for infinitely many $n$, in particular for monomial functions.

Later on, improvements of such a method have been used to answer existence questions about  several famous functions defined over finite fields which are also quite rare; see \cite{jedlicka_apn_2007,hernando_proof_2011}. The same approach has been applied in \cite{hernando_proof_2012} to prove a conjecture on monomial hyperovals and in \cite{leducq_functions_PN_2015} to get partial results towards the classification of monomial planar functions for infinitely many $n$, which was later completely solved by Zieve~\cite{zieve_planar_2015} by using the classification of indecomposable exceptional (permutation) polynomials. Similar results and approaches can also be found in \cite{caullery_large_class_2014,caullery_classification_2015,caullery_exceptional_2016,rodier_functions_APN_2011,schmidt_planar_2014,BZ}.

As in \cite{BZ}, the main tool is the use  of branches and  local quadratic transformations of a plane curve  to obtain a better estimate for the intersection number of two components of a fixed curve at one of its  singular points. Recently, an approach based on local quadratic transformations which uses implicitly branches has been applied in \cite{BS2018} to classify exceptional planar functions in characteristic two.

Consider an algebraic curve $\mathcal{C}$ defined over $\mathbb{F}_q$. Suppose, by way of contradiction, that $\mathcal{C}$ has no absolutely irreducible components over $\mathbb{F}_{q}$. We divide our proof into four steps.

\begin{enumerate}
    \item We find all the singular points of $\mathcal{C}$.
    \item We assume that $\mathcal{C}$ splits into two components $\mathcal{A}$ and $\mathcal{B}$ sharing no common irreducible component. An upper bound on the total intersection number of $\mathcal{A}$ and $\mathcal{B}$ is then obtained. The main ingredient here will be branch investigation using quadratic transformations.
    \item Under the assumption that $\mathcal{C}$ has no absolutely irreducible components over $\mathbb{F}_{q}$, we decompose $F(X,Y)$ as  $A(X,Y) B(X,Y)$ and obtain a lower bound on $(\deg A)( \deg B)$.
    \item Finally, by using B\'ezout's Theorem (see Theorem \ref{th:bezout}), we get a contradiction between the two bounds.
\end{enumerate}

\begin{theorem}[B\'ezout's Theorem]\label{th:bezout}
	Let $\mathcal{A}$ and $\mathcal{B}$ be two projective plane curves over an algebraically closed field $\mathbb{K}$, having no components in common. Let $A$ and $B$ be the polynomials associated with $\mathcal{A}$ and $\mathcal{B}$ respectively. Then
	\[
	\sum_P I(P, \mathcal{A}\cap \mathcal{B})=(\deg A)(\deg B),
	\]
	where the sum runs over all points in the projective plane $PG(2,\mathbb{K})$.
\end{theorem}

The following technical results will be used to study the branches at singular points of the curve $\mathcal{C}_f$.

\begin{proposition}\label{branch} Let $\mathcal{C}$ be the curve defined given by  $F(X,Y)=0$, where 
\begin{equation} \label{curvaprop}
F(X,Y)= AX^m+BY^n+ \sum a_{ij}X^iY^j,
\end{equation}
with $n<m$, $a_{m0}a_{0n}\neq0$, and
\begin{equation}\label{C:Conditions}
a_{ij}=0\qquad  \textrm{ if } \qquad \left\{ \begin{array}{l}
0<i<m; \textrm{ or}\\
i=0, j\leq n.
\end{array}\right.
\end{equation}
If $p\nmid (n,m)$ then $\mathcal{C}$ has $(n,m)$  branches centered at the origin.
\end{proposition}

\begin{proof}
If $n\mid m$ then, after applying $m_1=m/n-1$ times $F\mapsto F_1(X,Y)=F(X,XY)/X^n$ we can easily see that the origin is the center of $n=(n,m)$ distinct branches, since the tangent cone in $F_1(X,Y)$ is $AX^n+BY^n$.

Suppose now that $n\nmid m$. Let us consider $\ell_1$ the smallest integer such that  $m_1=m- \ell_1 n <n$. We apply $\ell_1$ times the local quadratic transformation  $F\mapsto F_1(X,Y)=F(X,XY)/X^n$. We have
\begin{equation*} \label{curvaprop1}
F_1(X,Y)= AX^{m_1}+BY^n+ \sum a_{ij}X^{i+\ell_1 (j-n)}Y^j.
\end{equation*}
By Conditions \eqref{C:Conditions} it is readily seen that the degree of each monomial $a_{ij}X^{i+\ell_1 (j-n)}Y^j$ is larger than $m_1$. Also, all the branches centered at the origin in $\mathcal{C}$ are still centered at the origin in $F_1(X,Y)$.

Apply now $k_1$ times the transformation  $G\mapsto G(XY,Y)/Y^{m_1}$, where $k_1$ is the smallest integer such that $n_1=n-k_1m_1\leq m_1$.

We distinguish two cases.
\begin{enumerate}
    \item $m_1\mid n$. In this case $F_2(X,Y)=F_1(XY,Y)/Y^{m_1}=AX^{m_1}+BY^{m_1}+\cdots$ and there are exactly $m_1=(n,m)$ branches centered at the origin in $\mathcal{C}$.
    \item $m_1\nmid n$. Then 
    \begin{equation*} \label{curvaprop2}
F_2(X,Y)= AX^{m_1}+BY^{n_1}+ \sum a_{ij}X^{i+\ell_1 (j-n)}Y^{j+(i+\ell_1(j-n)-m_1)k_1}.
\end{equation*}
Note that $i+\ell_1 (j-n)=0$ implies $i=0$ and $j=n$ and so $a_{ij}=0$ and so no monomial $Y^{\alpha}$ appears in $F_2(X,Y)$ apart from $BY^{n_1}$. Also $i+\ell_1 (j-n)<m_1$ if and only if $i+\frac{m-m_1}{n} (j-n)<m_1$ which yields $i+\frac{m-m_1}{n}j<m$ and so $i<m$. Since $a_{ij}=0$ if $0<i<m$, there is no monomial in $F_2(X,Y)$ with degree in $X$ smaller than $m_1$ apart from $BY^{n_1}$. Finally, all the branches centered at the origin in $F_1(X,Y)=0$ are centered at the origin in $F_2(X,Y)=0$.
\end{enumerate}
The polynomial $F_2(X,Y)$ satisfies Conditions \eqref{C:Conditions} and we can proceed by induction. 
\end{proof}

\begin{proposition}\label{branch2}
Let $\mathcal{C}$ be a curve of the affine equation 
$$Y^{q}+ \alpha X^{q}+X^{q^r-q^{r-1}+q-1}Y+L(X,Y), $$ 
where all the monomials in $L(X,Y)$ have degree at least $q^{r+1}+q-1$. Then there is a unique branch centered at the origin.
\end{proposition}
\proof
By induction on $r$.\\
If $r=1$, after applying the transformation $(X,Y) \mapsto (X,aX+Y)$, where $a^q+\alpha=0$, and $\theta(F(X,Y))=F(X,XY)/X^{q}$ one gets
$$Y^{q}+ a X^{q-1}+X^{q-1}Y+L^{\prime}(X,Y),$$
where $X\mid L^{\prime}(X,Y)$ and  $L^{\prime}(X,Y)$ contains monomials of degree at least  $q^2-1$. After applying $\eta(F(XY,Y)/Y^{q-1}$ one gets
$Y+aX^{q-1}+L^{\prime\prime}(X,Y)$,
with $Y\mid L^{\prime\prime}(X,Y)$, and therefore there is a unique branch centered at the origin.\\
Suppose that $r>1$. One applies $(X,Y) \mapsto (X,aX+Y)$, where $a^{q}+\alpha=0$, and    $q^{r-1}-q^{r-2}$ times  $\theta(F(X,Y))=F(X,XY)/X^{q}$, obtaining 
$$Y^q+aX^{q}+X^{q^{r-1}-q^{r-2}+q-1}Y+L^{\prime}(X,Y),$$
with monomials in $L^{\prime}(X,Y)$ of degree at least $q^{r+1}-q^r+q^{r-1}+q-1\geq q^{r}+q-1$ and the claim follows by the induction.
\endproof

\section{Exceptional scattered polynomials of index $t$}
In this section we investigate curves arising from index $t>0$ scattered polynomials. We assume that Conditions [C1], [C2], [C3] hold.

In the following it will be useful to consider the homogenized version of the starting linearized polynomial $f(X)\in \mathbb{F}_{q^r}[X]$. We denote it by the same symbol $f(X,T)$. Namely
$$f(X,T)=\sum_{i=0}^{M} A_i X^{q^{k_i}}T^{q^{k_M}-q^{k_i}}\in \mathbb{F}_{q^r}[X,T],$$
where $k_0=0$, $A_0 \ne 0$ and $A_M=1$. 

Recall that to each polynomial $f(X)$ is associated a curve $\mathcal{C}_f$ as shown in Lemma \ref{le:link}. In order to  apply the machinery described in Section \ref{Sec:Machinery}, we will investigate the singular points of the curve  $\mathcal{D}_f$ defined by $f(X)Y^{q^t}-f(Y)X^{q^t}=0$. In fact, as it can be easily seen, the set of its singular points contains also the singular points of $\mathcal{C}_f$. 
A homogeneous equation of $\mathcal{D}_f$ is given by $F(X,Y,T)=0$, where 
\begin{equation}\label{Def:F}
F(X,Y,T)=f(X,T)Y^{q^t}-f(Y,T)X^{q^t}=\sum_{i=0}^{M} A_i \left(X^{q^{k_i}}T^{q^{k_M}-q^{k_i}}Y^{q^t}-Y^{q^{k_i}}T^{q^{k_M}-q^{k_i}}X^{q^t}\right).
\end{equation}
The are no affine singular points in $\mathcal{D}_f$ apart from the origin. Note that the origin is, both in $\mathcal{C}_f$ and in $\mathcal{D}_f$, an ordinary singular point of multiplicity $q^{t}-q-1$ and $q^t$ respectively. The multiplicity of intersection of two putative components of $\mathcal{C}_f$ at such a point is therefore upperbounded by $(q^{t}-q-1)^2/4$.

All the other singular points of $\mathcal{D}_f$ (and therefore $\mathcal{C}_f$) are contained in the ideal line.

A singular point of $\mathcal{D}_f$ is the point $P=(0,1,0)$, while the other ideal singular points are of type $S_a=(a,1,0)$.

In order to study such points it is useful to consider the change of variables $(X,Y,T)\mapsto (T,Y,X)$. The affine equation of the corresponding curve $\widetilde{\mathcal{D}}_f$ is given by $G(X,Y)=0$, where 

\begin{equation}\label{Def:G}
    G(X,Y)=F(1,Y,X)=\sum_{i=0}^{M} A_i \left(X^{q^{k_M}-q^{k_i}}Y^{q^t}-Y^{q^{k_i}}X^{q^{k_M}-q^{k_i}}\right).
\end{equation}

Singular points of $\widetilde{\mathcal{D}}_f$   belong to three distinct groups:
\begin{itemize}
    \item $S_{\xi}=(0,\xi)$, with $\xi \in \mathbb{F}_{q}$.
    \item $R_{\xi}=(0,\xi)$, with $\xi^{q^{k_M}}=\xi^{q^t}$, $\xi \notin \mathbb{F}_q$ and $\xi^{q^{k_i}}\neq \xi^{q^t}$ for at least one  $i=0,\ldots,M-1$.
    \item $Q_{\xi}=(0,\xi)$, with $\xi^{q^{k_i}}=\xi^{q^t}$ for all $i=0,\ldots,M$ and $\xi \notin \mathbb{F}_q$.
\end{itemize}
The points $Q_{\xi}$ and $S_{\xi}$ are such that $\xi\in \mathbb{F}_{q^\ell}$ with $\ell=GCD(t,k_1-t,\ldots,k_M-t)=GCD(t,k_1,\ldots,k_M)$. These points are equivalent to $S_{1}$ via the automorphism $(X,Y)\mapsto (X,\xi Y)$.
While the point $S_1=(0,1)$ is equivalent to $S_0$ via the automorphism of $\tilde{\mathcal{D}}_f$ given by $(X,Y,T) \mapsto (X,Y+1,T)$. The point $S_0$ corresponds to the point $(1,0,0)$ of $\mathcal{D}_f$ which is equivalent to $P$ via the automorphism $(X,Y,T) \mapsto (Y,X,T)$. Hence we need to  study just the singularities $S_\xi$ and $R_\xi$ of $\tilde{\mathcal{D}}_f$.

First of all we consider singular points contained in the second group.

\begin{lemma} \label{lem1}
Let $R_{\xi}=(0,\xi)$, $\xi \in \mathbb{F}_{q^{k_M-t}}\setminus \mathbb{F}_{q^t}$, be a singular point of $\widetilde{ \mathcal{D}}_f$. If $k_i\geq t$ for each $i=1,\ldots,M$ then there is a unique branch centered at $R_{\xi}$. Thus, the multiplicity of intersection of two putative components of $\mathcal{C}_f$ in $R_{\xi}$ is $0$.

\end{lemma}
\proof
In order to study branches centered at $R_{\xi}$ we consider the polynomial 
\begin{eqnarray*}
H(X,Y)&=&G(X,Y+\xi)=G(X,Y)+G(X,\xi)=\sum_{i=0}^{M} A_i X^{q^{k_M}-q^{k_i}} \left(Y^{q^t}-Y^{q^{k_i}}+\eta_i\right)\nonumber \\
&=&Y^{q^t} -Y^{q^{k_M}}+B_0 X^{q^{k_M}-1}+\sum_{i=0}^{M-1} A_i X^{q^{k_M}-q^{k_i}} \left(Y^{q^t}-Y^{q^{k_i}}\right)+\sum_{i=1}^{M-1} B_i X^{q^{k_M}-q^{k_i}},
\end{eqnarray*}
where $\eta_i=\xi^{q^t}-\xi^{q^{k_i}}$ and $B_i=A_i \eta_i$. Note that, since $\xi\notin \mathbb{F}_{q^t}$, $B_0\neq 0$.

The point $R_{\xi}$ is mapped to the origin and its  tangent cone in $\mathcal{C}$ (the homogeneous polynomial  defined by $H(X,Y)=0$) is  $Y^{q^t}$.
In what follows we will perform a number of quadratic transformations.\\
Let $M_1$ be the largest index such that $B_{M_1}\neq 0$. Note that, since $R_{\xi}$ belongs to the second group, actually such $M_1$ exists. If $M_1=0$, then all $B_i=0$, $i=1,\ldots,M-1$.  We consider $q^{k_M-t}-1$ times the  transformation $\theta(H(X,Y))=H(X,XY)/X^{q^t}$ and we can easily see that 
$$H_1(X,Y)=Y^{q^t} -Y^{q^{k_M}} +B_0 X^{q^t-1}+L(X,Y).$$ The argument follows as in  Step 3.1.1. and  Step 3.1.2. below. Thus we consider now the case $M_1\neq 0$.

{\bf Step 1.} Let us consider the transformation $\theta(H(X,Y))=H(X,XY)/X^{q^t}$. Let 
$$u_1= \frac{q^{k_M}-q^{k_{M_1}}}{q^t}-1.$$


After $u_1$ applications of $\theta$, one gets 
\begin{eqnarray}\label{Def:H1}
H_1(X,Y)&=&Y^{q^t} -Y^{q^{k_M}} X^{u_1(q^{k_M}-q^t)} +B_{M_1}X^{q^t}+B_0 X^{q^{k_{M_1}}+q^t-1}+\sum_{i=0}^{M-1} A_i X^{q^{k_M}-q^{k_i}} Y^{q^t}\nonumber \\
&&-\sum_{i=0}^{M-1} A_i X^{q^{k_{M_1}}+q^t+(u_1-1)q^{k_i}} Y^{q^{k_i}}+\sum_{i=1}^{M_1-1} B_i X^{q^{k_{M_1}}+q^t-q^{k_i}}.
\end{eqnarray}

{\bf Step 2.} Let $\rho_1(X,Y)=(X,\alpha_1 X+Y)$ such that $\alpha_1^{q^t}+B_{M_1}=0$. After applying $\rho_1$ one gets 
\begin{eqnarray}\label{Def:H1primo}
H^{\prime}_1(X,Y)&=&Y^{q^t} {-(\alpha_1 X+Y)^{q^{k_{M}}} X^{u_1(q^{k_M}-q^t)}} +B_0 X^{q^{k_{M_1}}+q^t-1}+\sum_{i=0}^{M-1} A_i \alpha_1^{q^{t}} X^{q^{k_M}-q^{k_i}+q^t} \nonumber \\
&&+\sum_{i=0}^{M-1} A_i X^{q^{k_M}-q^{k_i}} Y^{q^t}-\sum_{i=0}^{M-1} A_i X^{q^{k_{M_1}}+q^t+(u_1-1)q^{k_i}} Y^{q^{k_i}}\nonumber\\
&&-\sum_{i=0}^{M-1} A_i\alpha_1^{q^{k_i}} X^{q^{k_{M_1}}+q^t+u_1q^{k_i}}+\sum_{i=1}^{M_1-1} B_i X^{q^{k_{M_1}}+q^t-q^{k_i}}.
\end{eqnarray}
Let us order the indices $k_i$ in such that $B_i \neq 0$ as $k_{M_1}>k_{M_2}>k_{M_3}>\cdots> k_{M_s}$. We distinguish two subcases.
\begin{enumerate}
\item Suppose that $M_2=0$. 

\noindent{\bf Step 3.1.1.} In $H_1^{\prime}(X,Y)$, the monomials of smallest degree are $Y^{q^t}$ and $B_0X^{q^{k_{M_1}}+q^t-1}$. If we apply $\theta$ exactly $q^{k_{M_1}-t}$ times we get 
$$\overline{H}(X,Y)=Y^{q^t}+B_0X^{q^t-1}+\overline{L}(X,Y),$$
where $L$ is a linearized polynomial in $Y$ with all the degrees in $X$ larger than $q^t-1$. Also, $0$ is the unique root of $\overline{H}(0,Y)$.\\

\noindent {\bf Step 3.1.2.} Now perform  $\tau(H(X,Y))=H(XY,X)/Y^{q^t-1}$: this gives 
$$\widetilde{H}(X,Y)=Y+B_0X^{q^t-1}+\widetilde{L}(X,Y),$$
where all the monomials in $\widetilde L$ have degree in $X$ larger than $1$ and then $0$ is the unique root of $\widetilde{H}(0,Y)$. The tangent cone has degree one now and there is a unique branch centered at the origin in the curve defined by $\widetilde{H}(X,Y)=0$ and so in $\mathcal{C}$. This also shows that the multiplicity of intersection of two putative components of $\mathcal{C}_f$ in the corresponding point is $0$.

\item Suppose now that $M_2\neq 0$.\\ 
First note that $q^{k_{M_1}}+q^t-q^{k_{M_2}}$ is the smallest degree of a monomial in $H_1^{\prime}$ apart from $Y^{q^t}.$ 

\noindent {\bf Step 3.2.1.} Let 
$$u_2=\frac{q^{k_{M_1}}-q^{k_{M_2}}}{q^t}.$$
We apply $u_2$ times   $\theta$  and get 
\begin{eqnarray*}
H_2(X,Y)&=&Y^{q^t}+B_0 X^{q^{k_{M_2}}+q^t-1}+\sum_{i=0}^{M-1} A_i \alpha_1^{q^{t}} X^{q^{k_M}-q^{k_{M_1}}+q^{k_{M_2}}-q^{k_i}+q^t} \nonumber \\
&&+\sum_{i=0}^{M-1} A_i X^{q^{k_M}-q^{k_i}} Y^{q^t}-\sum_{i=0}^{M-1} A_i X^{q^{k_{M_2}}+q^t+(u_1+u_2-1)q^{k_i}} Y^{q^{k_i}} \nonumber \\
&&-\sum_{i=0}^{M-1} A_i\alpha_1^{q^{k_i}} X^{q^{k_{M_2}}+q^t+u_1q^{k_i}}+\sum_{i=1}^{M_2} B_i X^{q^{k_{M_2}}+q^t-q^{k_i}}\nonumber\\
&&-(\alpha_1 X + X^{u_2}Y)^{q^{k_M}} X^{u_1(q^{k_M}-q^t)-u_2 q^t}.
\end{eqnarray*}

\noindent {\bf Step 3.2.2.} After $\rho_2(X,Y)=(X,\alpha_2 X+Y)$ with $\alpha_2^{q^t}+B_{M_2}=0$, one gets $H_2^{\prime}(X,Y)$ equal to 
\begin{eqnarray*}
&&Y^{q^t}+B_0 X^{q^{k_{M_2}}+q^t-1}+\sum_{i=0}^{M-1} A_i \alpha_1^{q^{t}} X^{q^{k_M}-q^{k_{M_1}}+q^{k_{M_2}}-q^{k_i}+q^t} \nonumber \\
&&+\sum_{i=0}^{M-1} A_i X^{q^{k_M}-q^{k_i}} (Y^{q^t}+\alpha_2^{q^t}X^{q^t})-\sum_{i=0}^{M-1} A_i X^{q^{k_{M_2}}+q^t+(u_1+u_2-1)q^{k_i}} (Y^{q^{k_i}}+\alpha_2^{q^{k_i}}X^{q^{k_i}})\nonumber\\
&&-\sum_{i=0}^{M-1} A_i\alpha_1^{q^{k_i}} X^{q^{k_{M_2}}+q^t+u_1q^{k_i}}+\sum_{i=1}^{M_2-1} B_i X^{q^{k_{M_2}}+q^t-q^{k_i}}\nonumber\\
&& -(\alpha_1 X + \alpha_2 X^{u_2+1}+X^{u_2} Y)^{q^{k_M}} X^{u_1(q^{k_M}-q^t)-u_2 q^t}.
\end{eqnarray*}
Now $H_2^{\prime}$ can be described as 
$$Y^{q^t}+B_0X^{q^{k_{M_2}}+q^t-1}+\sum_{i=1}^{M_2-1} B_i X^{q^{k_{M_2}}+q^t-q^{k_i}}+L(X,Y),$$
where $L(X,Y)$ is a linearized polynomial in $Y$ such that the monomials have degree  in $Y$ either $0$ or larger than $q^t-1$ and in $X$ larger than $k_{M_2}+q^t-1$. Note that also $H_1^{\prime}$ can be described in this way. 


\noindent {\bf Step 3.2.3.} We perform 
$$u_j=\frac{q^{k_{M_{j-1}}}-q^{k_{M_j}}}{q^t}$$
times $\theta$ and $\rho_j(X,Y)=(X,\alpha_j X+Y)$ with $\alpha_j^{q^t}+B_{M_j}=0$ and we obtain $$H_j^{\prime}(X,Y)=Y^{q^t}+B_0X^{q^{k_{M_j}}+q^t-1}+\sum_{i=1}^{M_j-1} B_i X^{q^{k_{M_j}}+q^t-q^{k_i}}+L^{\prime}(X,Y).$$
At the $s$-th step 
$$H_s^{\prime}(X,Y)=Y^{q^t}+B_0X^{q^{k_{M_s}}+q^t-1}+L^{\prime}(X,Y).$$
\noindent {\bf Step 3.2.4.} Note that at each step,   $0$ is the unique root of $H^{\prime}_j(0,Y)$ and therefore all the branches centered at the origin in $\mathcal{C}$ correspond to the branches centered at the origin in the curve defined by $H_s^{\prime}(X,Y)=0$. Another application of $u=q^{k_{M_s}}/q^t$
times $\theta$ gives
$$\overline{H}(X,Y)=Y^{q^t}+B_0X^{q^t-1}+\overline{L}(X,Y),$$
where $L$ is a linearized polynomial in $Y$ with all the degrees in $X$ larger than $q^t-1$. Now the assertion follows from point (1).
\end{enumerate}
\endproof

We now analyze the case in which $k_1=1$ and all the other $k_i\geq t$. Note that, using the same notation as in Lemma \ref{lem1}, $B_1\neq 0$, since $\xi\notin \mathbb{F}_q$.

\begin{lemma} \label{lem1_BIS}
Suppose $1=k_1<t<k_2<\cdots<k_M$, with $k_M\geq t+2$.
Let $R_{\xi}=(0,\xi)$, $\xi \in \mathbb{F}_{q^{k_M-t}}\setminus \mathbb{F}_{q^t}$, be a singular point of $\widetilde{ \mathcal{D}}_f$. Then there is a unique branch centered in $R_{\xi}$. Thus, the multiplicity of intersection of two putative components of $\mathcal{C}_f$ in $R_{\xi}$ is $0$.
\end{lemma}

\proof We proceed as in Lemma \ref{lem1}. Now $H(X,Y)=G(X,Y+\xi)=G(X,Y)+G(X,\xi)$ reads 
\begin{eqnarray*}
&&Y^{q^t} -Y^{q^{k_M}}+B_0 X^{q^{k_M}-1}+B_1 X^{q^{k_M}-q}\nonumber \\
&&+\sum_{i=0}^{M-1} A_i X^{q^{k_M}-q^{k_i}} \left(Y^{q^t}-Y^{q^{k_i}}\right)+\sum_{i=2}^{M-1} B_i X^{q^{k_M}-q^{k_i}},
\end{eqnarray*}
where $\eta_i=\xi^{q^t}-\xi^{q^{k_i}}$ and $B_i=A_i \eta_i$. Recall that, since $\xi\notin \mathbb{F}_{q^t}$, $B_0\neq 0$.\\
{\bf Case  $B_2=\cdots=B_{M-1}=0$}. \\
We perform $u=q^{k_M-t}-1$ transformations $\theta(H(X,Y))=H(X,XY)/X^{q^t}$ and we get 
$$H_1=Y^{q^t} +B_0 X^{q^{t}-1}+B_1 X^{q^{t}-q}+Y^{q^t}\sum_{i=0}^{M-1} A_i X^{q^{k_M}-q^{k_i}}-\sum_{i=0}^{M} A_i X^{q^t+q^{k_i}(q^{k_M-t}-2)}Y^{q^{k_i}}.$$
 Now we perform one time $\eta(H(X,Y))=H(XY,Y)/Y^{q^t-q}$ and we get 
$$H_2=Y^{q} +B_0 X^{q^{t}-1}Y^{q-1}+B_1 X^{q^{t}-q}+ \sum_{i=0}^{M-1} A_i X^{q^{k_M}-q^{k_i}}Y^{q^{k_M}-q^{k_i}+q}-\sum_{i=0}^{M} A_i X^{q^t+q^{k_i}(q^{k_M-t}-2)}Y^{q+q^{k_i}(q^{k_M-t}-1)}.$$
It is readily seen that  all the branches centered at the origin in $H_1=0$ are mapped to branches centered at the origin in $H_2=0$.
Apply $v=q^{t-1}-2$ times $\theta(H(X,Y))=H(X,XY)/X^{q}$ obtaining 
$$H_3=Y^{q} +B_0 X^{q^{t}-q^{t-1}+1}Y^{q-1}+B_1 X^{q}+ L(X,Y).$$
Now $H_4=H_3(X,\alpha_1X+Y)$, where $\alpha_1^q+B_1=0$, reads 
$$H_4=Y^q+B_0X^{q^t-q^{t-1}+1}(\alpha_1X+Y)^{q-1}+L_2(X,\alpha_1X+Y).$$
All the monomials in $L(X,\alpha_1X+Y)$ have degree at least $q^{t+1}+q^t-q^{t-1}+q-1$ (consider the case $k_M=t+2$ and $i=0$). After $w=q^{t-1}-q^{t-2}$ applications of 
$\theta(H(X,Y))=H(X,XY)/X^{q}$
we have 
$$H_5=Y^q+B_0X^q-B_0\alpha_1^{q-2}X^{q^{t-1}-q^{t-2}+q-1}Y+\cdots$$
where the other terms have degree in $X$ at least $q^{t+1}+q-1$. By Proposition \ref{branch2} there is a unique branch centered at the origin.\\ 
{\bf Case  $B_i\neq 0$ for some $i>1$}. \\
Let $M_1=\max\{i>1 :  B_i\neq 0\}$. Such $M_1$ is well defined. We now consider steps as in the proof of Proposition \ref{lem1_BIS}. The main difference here is the presence of the monomial  $B_1X^{q^{k_M}-q}$. After 
{\bf Step 1.} this monomial is mapped to $B_1X^{q^{k_{M_1}}-q}$ and it is fixed by {\bf Step 2}. If $M_2=0$ then we use the same approach as in {\bf Case  $B_2=\cdots=B_{M-1}=0$} and the claim follows. Suppose now $M_2\neq 0$. We perform {\bf Step 3.2.1}, {\bf Step 3.2.2}, {\bf Step 3.2.3}, and {\bf Step 3.2.4}: $B_1X^{q^{k_{M_1}}-q}$ becomes $B_1X^{q^{k_{M_s}}-q}$. Now we proceed as in {\bf Case  $B_2=\cdots=B_{M-1}=0$} and the claim follows.
\endproof

\begin{lemma} \label{lem1_2}
Let $R_{\xi}=(0,\xi)$, $\xi \in  \mathbb{F}_{q^t}$, be a singular point of $\widetilde{ \mathcal{D}}_f$. If $k_i\geq t$ for each $i=1,\ldots,M$ then there is a unique branch centered in $R_{\xi}$. Thus, the multiplicity of intersection of two putative components of $\mathcal{C}_f$ in $R_{\xi}$ is $0$.

\end{lemma}
\proof
We proceed as in Lemma \ref{lem1_2}. The difference here is that $B_0=0$. Also, let $$j=\max\{i=1,\ldots,M-1 \ : \ B_i\neq 0\}.$$ Note that $B_j$ is well defined. So 
\begin{eqnarray}\label{Def:H_2}
H(X,Y)&=&G(X,Y+\xi)=G(X,Y)+G(X,\xi)=\sum_{i=0}^{M} A_i X^{q^{k_M}-q^{k_i}} \left(Y^{q^t}-Y^{q^{k_i}}+\eta_i\right)\nonumber \\
&=&Y^{q^t} -Y^{q^{k_M}}+\sum_{i=0}^{M-1} A_i X^{q^{k_M}-q^{k_i}} \left(Y^{q^t}-Y^{q^{k_i}}\right)+B_j X^{q^{k_M}-q^{k_j}} +\sum_{i=1}^{j-1} B_i X^{q^{k_M}-q^{k_i}},
\end{eqnarray}
where $\eta_i=\xi^{q^t}-\xi^{q^{k_i}}$ and $B_i=A_i \eta_i$.

We apply $u_1=\frac{q^{k_M}-q^{k_j}}{q^t}-1$ times the transformation $F(X,Y)\mapsto F(X,XY)/X^{q^t}$ and then

\begin{eqnarray}\label{Def:H1_2}
H_1(X,Y)&=&Y^{q^t} -Y^{q^{k_M}} X^{u_1(q^{k_M}-q^t)} +B_{j}X^{q^t}+\sum_{i=0}^{M-1} A_i X^{q^{k_M}-q^{k_i}} Y^{q^t}\nonumber \\
&&+A_0X^{q^{k_j}+q^t+q^{k_M-t}-q^{k_j-t}-2}Y-\sum_{i=1}^{M-1} A_i X^{q^{k_{j}}+q^t+(u_1-1)q^{k_i}} Y^{q^{k_i}}+\sum_{i=1}^{j-1} B_i X^{q^{k_{j}}+q^t-q^{k_i}}.
\end{eqnarray}
Let $\rho_1(X,Y)=(X,\alpha_1 X+Y)$ such that $\alpha_1^{q^t}+B_{j}=0$. After applying $\rho_1$ one gets 
\begin{eqnarray}\label{Def:H1primo_2}
H^{\prime}_1(X,Y)&=&Y^{q^t} +\sum_{i=1}^{j-1} B_i X^{q^{k_{j}}+q^t-q^{k_i}}+A_0 X^{q^{k_j}+q^t+q^{k_M-t}-q^{k_j-t}-1}\\ \nonumber
&&+A_0 X^{q^{k_j}+q^t+q^{k_M-t}-q^{k_j-t}-2}Y+L(X,Y),
\end{eqnarray}
where $L(X,Y)$ is a polynomial containing only terms of degree in $X$ larger than $q^{k_j}+q^t+q^{k_M-t}-q^{k_j-t}-1$. Note that $q^t\mid (q^{k_{j}}+q^t-q^{k_i})$. Suppose that the indices $i$ such that $B_i\neq0$ are ordered as
$$i_1<i_2<\cdots<i_s=j.$$
We continue performing each time $$\frac{q^{k_{i_\ell}}-q^{k_{i_{\ell-1}}}}{q^t}$$
times $F(X,Y)\mapsto F(X,XY)/X^{q^t}$ and $\rho_1(X,Y)=(X,\alpha_1 X+Y)$. Doing so, in a similar way as in Lemma \ref{lem1} we obtain 
$$\widetilde H(X,Y)=Y^{q^t}+A_0X^{\beta}+\cdots$$
where $(\beta,q^t)=1$. We now apply Proposition \ref{branch} and we deduce that there is a unique branch centered at the origin. 
\endproof

This completes the analysis of the points $R_\xi$ for $k_i\geq t$, $i=1,\ldots,M$. The following proposition will be used to study the point $S_1$.


\begin{lemma} \label{lem1_2BIS}
Suppose $1=k_1<t<k_2<\cdots<k_M$, with $k_M\geq t+2$. Let $R_{\xi}=(0,\xi)$, $\xi \in  \mathbb{F}_{q^t}$, be a singular point of $\widetilde{ \mathcal{D}}_f$. Then there is a unique branch centered in it. Thus, the multiplicity of intersection of two putative components of $\mathcal{C}_f$ in the corresponding point is $0$.
\end{lemma}
\proof
Recall that since $\xi\notin \mathbb{F}_q$, $B_1\neq 0$. The proof is the same as the one in Proposition \ref{lem1_BIS}, since $B_0=0$ does not affect the computations.  
\endproof

The following lemma deals with the points $S_\xi$ (and therefore with the points $Q_\xi$). Here we do not assume that $k_i>t$ for $i>0$.
\begin{lemma} \label{lem2}
Let $S_1=(0,1)\in \widetilde{\mathcal{D}}_f$ and $k_M \geq t$.
\begin{itemize}
    \item If $t\mid k_M$ then there are $q^t$ branches centered at $S_1$.
    \item If $k_M=tr+s$, with $s\in \{1,\ldots,t-1\}$ then there are $q^{(s,t)}$ branches centered at $S_1$.
\end{itemize}
The maximum possible intersection multiplicity of two components of $\widetilde{\mathcal{D}}_f$ at $S_1$ is $\frac{q^{k_M+t}}{4}.$
\end{lemma}
\begin{proof}
Following the same notations as in Lemma \ref{lem1}, 
\begin{eqnarray*}
H(X,Y)&=&Y^{q^t}-Y^{q^{k_M}}+Y^{q^t}\sum_{i=0}^{M-1}A_iX^{q^{k_M}-q^{k_i}}-\sum_{i=0}^{M-1}A_iX^{q^{k_M}-q^{k_i}}Y^{q^{k_i}}.
\end{eqnarray*}
We distinguish two cases.
\begin{itemize}
    \item $t\mid k_M$. We perform $(q^{k_M}-1)/(q^t-1)-1$ times $\theta$ and we get
    $$\widetilde{H}(X,Y)=Y^{q^t}-A_0YX^{q^t-1}+\cdots.$$
Hence there are $q^t$ branches at the $q^t$-singular point $S_1$. All the branches centered at the origin in $\mathcal{C}: H(X,Y)=0$ are
    $$Z_i=\left(t,\eta_i t^{\alpha}+\delta \right),$$
    where $\alpha=(q^{k_M}-1)/(q^{t}-1)$, $\eta_i\in \mathbb{F}_{q^{k_M-t}}^*$, and  $\deg_t(\delta)>\alpha$. Suppose now that the curve $\mathcal{C}$ splits into two components $\mathcal{X}$ and $\mathcal{Y}$ sharing no common irreducible component. It follows that $\mathcal{X}$ and $\mathcal{Y}$ have no branches in common. Let $U(X,Y)$ and $V(X,Y)$ be two polynomials defining the components $\mathcal{X}$ and $\mathcal{Y}$ such that $U$ and $V$ have no common factors. Then
$$U(X,Y)=Y^m+U_0(X,Y), \ {\rm and} \ V(X,Y)=Y^{q^t-m}+V_0(X,Y),$$
where $0 \leq m \leq q^t$, $\deg(U_0)>m$ and $\deg(V_0)>q^t-m$.

Our aim is to compute the intersection multiplicity of $\mathcal{X}$ and $\mathcal{Y}$ at the origin. For a  branch $Z_i$ contained in $\mathcal{X}$ it follows that the coefficient of the term  degree $m\alpha$ (in $t$) in $U(Z_i)$ vanishes, that is, 
$$\eta_i^m+\sum_{r=0}^{m} \gamma_r\eta_i^{m-r}=0,$$
where the monomials  $\gamma_rX^{r\alpha}Y^{m-r}$ belong to $U(X,Y)$. Analogously, if a branch $Z_j$ belongs to $\mathcal{Y}$ then  
$$\eta_i^{q^t-m}+\sum_{r=0}^{q^t-m} \bar{\gamma}_r\eta_i^{q^t-m-r}=0,$$
where the monomials  $\overline{\gamma}_rX^{r\alpha}Y^{q^t-m-r}$ belong to $V(X,Y)$,
since the coefficient of the term  degree $(q^t-m)\alpha$ (in $t$) in $V(Z_j)$ vanishes. Therefore, since  $\eta_i,\eta_j\neq 0$, and there are exactly $q^{t}-1$ branches corresponding to distinct $\eta_i$, exactly $m$ of them belong to $\mathcal{X}$ and $q^t-m$ to $\mathcal{Y}$. Hence if $Z_i$ belongs to $\mathcal{X}$, then $Z_i$ does not belong to $\mathcal{Y}$. So the multiplicity of intersection at the origin of two putative components of $\mathcal{C}$ is given by
$$(q^t-m)m \beta\leq \frac{q^{2t}}{4}q^{k_M-t}=\frac{q^{k_M+t}}{4}.$$
    
    \item $t\nmid k_M$. Let $k_M=tr+s$, with $s\in \{1,\ldots,t-1\}$. We perform $q^s(q^{k_M-s}-1)/(q^t-1)$ times $\theta$ and we get
    $$\widetilde{H}(X,Y)=Y^{q^t}-A_0YX^{q^s-1}+X^{q^t}Y^{q^t}L(X,Y)\cdots,$$
    for some $L(X,Y)$. Apart from the branch with tangent line $Y=0$, the other branches centered at the origin correspond to the branches centered at the origin for 
    $$Y^{q^t-1}-A_0X^{q^s-1}+X^{q^t}Y^{q^t-1}L(X,Y)\cdots=0.$$
    By Proposition  \ref{curvaprop} 
     there are other $q^{(s,t)}-1$ branches.
    All the branches centered at the origin in $\mathcal{C}: H(X,Y)=0$ are
    $$Z_i=\left(t^{\alpha},\eta_i t^{\beta}+\delta \right),$$
    where $\alpha=(q^{t}-1)/(q^{(t,s)}-1)$, $\beta=(q^{k_M}-1)/(q^{(t,s)}-1)$, $\eta_i\in \mathbb{F}_{q^{(s,t)}}^*$, and  $\deg_t(\delta)>\beta$. Suppose now that the curve $\mathcal{C}$ splits into two components $\mathcal{X}$ and $\mathcal{Y}$ sharing no common irreducible component. It follows that $\mathcal{X}$ and $\mathcal{Y}$ have no branches in common. Let $U(X,Y)$ and $V(X,Y)$ be two polynomials defining the components $\mathcal{X}$ and $\mathcal{Y}$ such that $U$ and $V$ have no common factors. Then
$$U(X,Y)=Y^m+U_0(X,Y), \ {\rm and} \ V(X,Y)=Y^{q^t-m}+V_0(X,Y),$$
where $0 \leq m \leq q^t$, $\deg(U_0)>m$ and $\deg(V_0)>q^t-m$.

Our aim is to compute the intersection multiplicity of $\mathcal{X}$ and $\mathcal{Y}$ at the origin. For a  branch $Z_i$ contained in $\mathcal{X}$ it follows that the coefficient of the term  degree $m\beta$ (in $t$) in $U(Z_i)$ vanishes, that is, 
$$\eta_i^m+\sum_{r=0}^{\lfloor m/\alpha\rfloor} \gamma_r\eta_i^{m-r\alpha}=\eta_i^{m-r\lfloor m/\alpha\rfloor}\sum_{r=0}^{\lfloor m/\alpha\rfloor} \gamma_{\lfloor m/\alpha\rfloor-r}\eta_i^{r\alpha}=\eta_i^{m-r\lfloor m/\alpha\rfloor}\sum_{r=0}^{\lfloor m/\alpha\rfloor} \gamma_{\lfloor m/\alpha\rfloor-r}\left(\eta_i ^{\alpha}\right)^{r}=0,$$
where the monomials  $\gamma_rX^{r\beta}Y^{m-r\alpha}$ belong to $U(X,Y)$. Analogously, if a branch $Z_j$ belongs to $\mathcal{Y}$ then  
$$\eta_j^{q^t-m-r\lfloor (q^t-m)/\alpha\rfloor}\sum_{r=0}^{\lfloor (q^t-m)/\alpha\rfloor} \overline{\gamma}_{\lfloor (q^t-m)/\alpha\rfloor-r}\left(\eta_j^{\alpha}\right)^{r}=0,$$
where the monomials  $\overline{\gamma}_rX^{r\beta}Y^{q^t-m-r\alpha}$ belong to $V(X,Y)$,
since the coefficient of the term  degree $(q^t-m)\beta$ (in $t$) in $V(Z_j)$ vanishes. Therefore, since  $\eta_i,\eta_j\neq 0$, and there are exactly $q^{(s,t)-1}$ branches corresponding to distinct $\eta_i$, exactly $\lfloor m/\alpha\rfloor$ of them belong to $\mathcal{X}$ and $\lfloor (q^t-m)/\alpha\rfloor$ to $\mathcal{Y}$. In particular, noting that $X^\alpha$ is a permutation of $\mathbb{F}_{q^{(s,t)}}$, $Z_i$ belongs to $\mathcal{X}$ if and only if $$G(\eta_i)=\sum_{r=0}^{\lfloor m/\alpha\rfloor} \gamma_{\lfloor m/\alpha\rfloor-r}\eta_i^{r}=0$$ and to $\mathcal{Y}$ if and only if $$\overline{G}(\eta_i)=\sum_{r=0}^{\lfloor (q^t-m)/\alpha\rfloor} \overline{\gamma}_{\lfloor (q^t-m)/\alpha\rfloor-r}\eta_i^{r}=0.$$ Suppose now that $Z_i$ belongs to $\mathcal{X}$, then $\overline{G}(\eta_i)\neq 0$ and the coefficient of the term in $t$ of degree $(q^t-m)\beta$ in $V(Z_i)$ does not vanish. So the multiplicity of intersection at the origin of two putative components of $\mathcal{C}$ is given by
$$(q^t-m)\left\lfloor \frac{m}{\alpha}\right\rfloor \beta\leq \frac{q^{2t}}{4}q^{k_M-t}=\frac{q^{k_M+t}}{4}.$$
\end{itemize}

\end{proof}

\begin{proposition} \label{fine}
Let $t \geq 2$ be a natural number, $f(X)=\sum_{i=0}^{M}A_i X^{k_i} \in \mathbb{F}_{q^r}[X]$ where $k_0=0$, $A_M=1$ and $k_i \geq t$ for $i \geq 2$. Let $\mathcal{C}_f$ be the algebraic curve associated with $f$ as in Lemma \ref{le:link}. If $t \mid k_M$ and $k_M \geq 3t$ or $t \nmid k_M$ and $k_M \geq 2t-1$ then $\mathcal{C}_f$ has an absolutely irreducible component defined over $\mathbb{F}_{q^r}$. In particular, $f(X)$ is not exceptional scattered.
\end{proposition} 

\begin{proof} Suppose that $\mathcal{C}_f: \tilde{F}(X,Y)=0$ splits into two components $\mathcal{X}$ and $\mathcal{Y}$ sharing no absolutely irreducible component. Clearly their intersection points are singular points of $\mathcal{C}_f$. As previously observed, the origin is an ordinary singular point of $\mathcal{C}_f$ of multiplicity $q^{t}-q-1$ and from Lemmas \ref{lem1}, \ref{lem1_BIS}, \ref{lem1_2} and \ref{lem1_2BIS}, $I(R_\xi,\mathcal{X} \cap \mathcal{Y})=0$. Let $T \in \mathcal{I}:=\{P,S_\xi,Q_\xi\}$. From Lemma \ref{lem2}  $I(T,\mathcal{X} \cap \mathcal{Y}) \leq q^{k_M+t}/4$. Note that $|\mathcal{I}|=1+q+(q^\ell-q)=q^\ell+1$, where $\ell=\gcd(t,k_1,\ldots,k_M)$. 
Hence, if $t \mid k_M$ then
    \begin{equation} \label{int1}
    \sum_{T \in \mathcal{X} \cap \mathcal{Y}} I(T,\mathcal{X} \cap \mathcal{Y}) \leq \frac{(q^t-q-1)^2}{4}+(q^t+1)\frac{q^{k_M+t}}{4};
    \end{equation}
    while if $t \nmid k_M$ then
     \begin{equation} \label{int2}
     \sum_{T \in \mathcal{X} \cap \mathcal{Y}} I(T,\mathcal{X} \cap \mathcal{Y}) \leq \frac{(q^t-q-1)^2}{4}+(q^{t/2}+1)\frac{q^{k_M+t}}{4}.
     \end{equation}
Assume that $\tilde{F}(X,Y)=W_1(X,Y)\ldots W_k(X,Y)$ is the decomposition of $\tilde{F}(X,Y)$ over $\mathbb{F}_{q^r}$ with $\deg(W_i)=d_i$ and $\sum_{i=1}^{k}d_i=q^{k_M}+q^t-q-1=\deg(\tilde{F}(X,Y))$ and suppose by contradiction that $\mathcal{C}_f$ has no absolutely irreducible components defined over $\mathbb{F}_{q^r}$. 
From \cite[Lemma 10]{hernando_proof_2011}, there exist natural numbers $s_i$ such that $W_i$ splits into $s_i$ absolutely irreducible factors over $\bar{\mathbb{F}}_{q^r}$ each of degree $d_i/s_i$. Since $\mathcal{C}_f$ has no absolutely irreducible factors defined over $\mathbb{F}_{q^r}$, $s_i >0$ for $i=1,\ldots,k$.
Consider the polynomials
$$A(X,Y)=\prod_{i=1}^{k} \prod_{j=1}^{\lfloor s_i/2 \rfloor} Z_i^{j}(X,Y), \ B(X,Y)=\prod_{i=1}^{k}\prod_{j=\lfloor s_i/2 \rfloor+1}^{s_i} Z_i^j(X,Y),$$
where $Z_i^1(X,Y),\ldots,Z_i^{s_i}(X,Y)$ are absolutely irreducible components of $W_i(X,Y)$.
Let $\alpha$ and $\alpha+\beta$ be the degrees of $A(X,Y)$ and $B(X,Y)$ respectively. 
Then $2\alpha+\beta=\deg(\mathcal{C}_f)=q^{k_M}+q^t-q-1$, $\beta \leq \alpha$ and $\beta \leq (q^{k_M}+q^t-q-1)/3$. Furthermore from $\alpha=(q^{k_M}+q^t-q-1-\beta)/2$,
$$\deg(A)\deg(B)=\alpha(\alpha+\beta)= \frac{(q^{k_M}+q^t-q-1)^2-\beta^2}{4} \geq \frac{2(q^{k_M}+q^t-q-1)^2}{9}.$$
Let $\mathcal{A}:A(X,Y)=0$ and $\mathcal{B}: B(X,Y)=0$. By B\'ezout’s Theorem \ref{th:bezout},
\begin{equation} \label{bez}
\sum_{T \in \mathcal{A} \cap \mathcal{B}}I(T,\mathcal{A} \cap \mathcal{B})=\deg(A)\deg(B) \geq \frac{2(q^{k_M}+q^t-q-1)^2}{9}.
\end{equation}
Now we can combine \eqref{bez} with \eqref{int1} and \eqref{int2}.
Assume first that $t \mid k_M$ so that $k_M=\gamma t$, $\gamma \geq 1$. Then from \eqref{bez} and \eqref{int1} we get
$$\frac{(q^t-q-1)^2}{4}+(q^t+1)\frac{q^{k:M+t}}{4} \geq \frac{2(q^{\gamma t}+q^t-q-1)^2}{9},$$
which is false whenever $\gamma \geq 3$.
If $t \nmid k_M$ then write $k_M=\gamma t+s$ with $s =1,\ldots,t-1$. From \eqref{bez} and \eqref{int2}
$$\frac{(q^t-q-1)^2}{4}+(q^{t/2}+1)\frac{q^{(\gamma+1)t+s}}{4} \geq \frac{2(q^{\gamma t+s}+q^t-q-1)^2}{9},$$
which is false whenever $k_M \geq 2t-1$.
This shows that $\mathcal{C}_f$ has an absolutely irreducible component defined over $\mathbb{F}_{q^r}$. To show that $f(X)$ is not exceptional scattered it is sufficient to show that if $r$ is sufficiently large then $\mathcal{C}_f$ has an affine point $P=(x,y)$ with $x/y \not\in \mathbb{F}_q$. From the Hasse-Weil bound
$$|\mathcal{C}_f(\mathbb{F}_{q^r})| \geq q^r+1-(q^{k_M}+q^t-q-2)(q^{k_M}+q^t-q-3)\sqrt{q^r}.$$
The number of ideal points of $\mathcal{C}_f$ is at most $q^{k_M-t}$, while the number of affine points $P=(x,y)$ of $\mathcal{C}_f$ with $x/y \in \mathbb{F}_q$ is at most $q(q^{k_M}+q^t-q-1)$. Hence it is sufficient to observe that
$$q^r+1-(q^{k_M}+q^t-q-2)(q^{k_M}+q^t-q-3)\sqrt{q^r}-q^{k_M-t}-q(q^{k_M}+q^t-q-1) >0$$
for $r \geq 5k_M$.
\end{proof}

We note that for $t=2$ the hypothesis $k_i \geq 2$ for $i \geq 2$ is trivially satisfied. Hence the complete classification of exceptional scattered polynomials of index $2$ follows as a corollary of Proposition \ref{fine}.
\begin{corollary}
The only exceptional scattered monic polynomials $f$ of index $2$ over $\mathbb{F}_{q^r}$ are those of type \eqref{eq:the_one}.
\end{corollary}

\section{Casi sporadici}

\begin{theorem}{\rm{(\cite[Remark 3.2 and Theorem 3.3]{BZ})}} \label{sporadic}
Let $q$ be a prime power. Let Then $X^{q^k}$ is the unique exceptional scattered monic polynomial of index $0$ unless one of the following cases occurs.
\begin{enumerate}
\item $q=2$ and $f(X)=X^{2^{k-1}}+bX^{2^{k}}$ with $b \in \mathbb{F}_{2^n}^*$; 
\item $q=2$ and $f(X)=X^{2^{k-2}}+aX^{2^{k-1}}+bX^{2^{k}}$ with $k \geq 7$ and $a,b \in \mathbb{F}_{2^n}^*$; 
\item $q=3$ and $f(X)=X^{3^{k-1}}+bX^{3^{k}}$ with  $b \in \mathbb{F}_{3^n}^*$;
\item $q=4$ and $f(X)=X^{4}+bX^{16}$ or $f(X)=X^{16}+bX^{64}$ with $b \in \mathbb{F}_{4^n}^*$;
\item $q=5$ and $f(X)=X^{5}+bX^{25}$ with $b \in \mathbb{F}_{5^n}^*$.
\end{enumerate}
\end{theorem} 

We are going to proceed with a cases-by-case analysis of 1-5 applying the following general strategy:

\begin{theorem}
Let $f(X)=X^{q^i}+bX^{q^{i+1}}$, $b\in \mathbb{F}_{q^n}$, a polynomial in $\mathbb{F}_{q^n}[X]$. Then $f(X)$ is not exceptional scattered.
\end{theorem}
\proof
Consider the curve $\mathcal{Y}_f$ associated with $f(X)$ defined by 
$$X^{q-1}=
-\frac{Y^{q^i-1}-1}{bY^{q^{i+1}-1}-1}=
-\frac{\prod_{\alpha \in \mathbb{F}_{q^i}^*}(Y-\alpha)}
{b\prod_{\beta \in \mathbb{F}_{q^{i+1}}^*}(Y-\beta)}=g(Y).$$

Let $F$ be the function field associated with $\mathcal{Y}_f$. Since it can be seen as a Kummer extension of the rational function field and all the zeros of $g(Y)$ are simple, by \cite[Proposition 3.7.3]{Sti} $\mathbb{F}_{q^n}$ is the constant field of $F$. If $n$ is large enough, since its genus is at most
$$1-(q-1)+\frac{(q^{i+1}+q^i-2q-2)(q-2)}{2}\leq \frac{q^{i+2}}{2}$$
(see again \cite[Proposition 3.7.3]{Sti}), the curve $\mathcal{Y}_f$  contains points of type $(x_0,y_0)\in \mathbb{F}_{q^n}^2$, where $y_0\notin \mathbb{F}_q$.

Consider now 
$$F(X,Y)=\frac{
(X^{q^i}+bX^{q^{i+1}})Y-(Y^{q^i}+bY^{q^{i+1}})X} {XY(X^{q-1}-Y^{q-1})}=\frac{
X^{q^i-1}+bX^{q^{i+1}-1}-Y^{q^i-1}-bY^{q^{i+1}-1}} {X^{q-1}-Y^{q-1}}.$$
The curve $\mathcal{X}_f$ has equation $F(X,Y)=0$. 
It means that the pair $(x_0^{q^{n-i}},y_0 x_0^{q^n-i})$ is such that \textcolor{red}{$F(x_0^{q^{n-i}},x_0^{q^n-i}y_0)=0$}\footnote{\`E vero? mi sto intrecciando con i conti, ho fatto solo le trasformazioni al contrario\ldots} and $x_0^{q^{n-i}}/x_0^{q^n-i}y_0=1/y_0 \notin \mathbb{F}_q$. So $f(X)$ is not exceptional scattered.
\endproof

\begin{theorem}
Let $q=2$. and consider $f(X)=X^{q^i}+aX^{q^{i+1}}+bX^{q^{i+2}}$, $a,b\in \mathbb{F}_{q^n}$, a polynomial in $\mathbb{F}_{q^n}[X]$. Then $f(X)$ is not exceptional scattered.
\end{theorem}
\proof
Consider the equation
$$F(X,Y)=0 \iff \frac{f(X)Y-f(Y)X}{X^qY-XY^q}=\frac{X^{q^i-1}-Y^{q^i-1}+a(X^{q^{i+1}-1}-Y^{q^{i+1}-1})+b(X^{q^{i+2}-1}-Y^{q^{i+2}-1})}{X^{q-1}y^{q-1}}=0.$$
With the substitution $(X,Y) \to (X,XY)$, and considering $XY\neq 0$, $Y^{q-1}\neq 1$, the above equation reads
$$X^{q^i-q}\frac{Y^{q^i-1}-1}{Y^{q-1}-1}+aX^{q^{i+1}-q}\frac{Y^{q^{i+1}-1}-1}{Y^{q-1}-1}+bX^{q^{i+2}-q}\frac{Y^{q^{i+2}-1}-1}{Y^{q-1}-1}=0
$$
that is 
\begin{equation}\label{EqCubica}\frac{Y^{q^{i}-1}-1}{Y^{q^{i+2}-1}-1}+aX^{q^{i+1}-q^i}\frac{Y^{q^{i+1}-1}-1}{Y^{q^{i+2}-1}-1}+bX^{q^{i+2}-q^i}=0 \iff \frac{Y^{q^{i}-1}-1}{Y^{q^{i+2}-1}-1}+aZ\frac{Y^{q^{i+1}-1}-1}{Y^{q^{i+2}-1}-1}+bZ^{q+1}=0,
\end{equation}
where $Z=X^{q^i(q-1)}$.

Recalling that $q=2$, in Equation \eqref{EqCubica}, the polynomial 
$$G(Y,Z)=\frac{Y^{q^{i}-1}-1}{Y^{q^{i+2}-1}-1}+aZ\frac{Y^{q^{i+1}-1}-1}{Y^{q^{i+2}-1}-1}+bZ^{q+1}\in \mathbb{F}_{q^n}(Y)[Z]$$
is absolutely irreducible if and only it does not contain a factor of type $Z-h(Y)/g(Y)$, with both $h,g\in \overline{\mathbb{F}_q}[Y]$ and coprime. If it happens, then
\begin{equation}\label{EqRidotta}
    bh^3(Y)(Y^{2^{i+2}-1}-1)+ah(Y)g^2(Y)(Y^{2^{i+1}-1}-1)+g^3(Y)(Y^{2^{i}-1}-1)
\end{equation}
vanishes. Consider now a root $\alpha$ of $g(Y)$. Then $(Y-\alpha)^2$ should divide $h^3(Y)(Y^{2^{i+2}-1}-1)$ and therefore $Y^{2^{i+2}-1}-1$, which is impossible since this polynomial is separable. So $g(Y)$ is constant. If $d=\deg(f)$ it is readily seen that the highest degree monomial in \eqref{EqRidotta} is $Y^{3d+2^{i+2}-1}$ and it does not vanish. 
This shows that $G(Y,Z)$ is absolutely irreducible and defines an absolutely irreducible curve defined over $\mathbb{F}_{q^n}$ curve $\mathcal{Y}_f$. This yields the existence of pairs $(x_0,y_0) \in \mathbb{F}_{q^n}^2$ with $x_0/y_0\notin \mathbb{F}_q$ and therefore $f(X)$ is not exceptional scattered.

\endproof

\section{The open cases for $t=0$}

In this subsection we prove that Theorem \ref{t01} (1) holds also for $q \leq 5$, that is for the open cases left in \cite{BZ}. Since the open cases regards binomials and trinomials  in the following  we analyze two families of binomials and trinomials in a more general setting. 

\subsection{General Binomials}

Consider a curve of type $\mathcal{X}_f : F(X,Y)=0$, where 
$$F(X,Y)=\frac{(X^{q^n}+bX^{q^m})Y^{q^t}-(Y^{q^n}+bY^{q^m})X^{q^t}}{X^qY-XY^q}\in \mathbb{F}_{q^k}(X,Y),$$
where $n<m$ and $b\in \mathbb{F}_{q^k}$.
Now,
\begin{eqnarray*}
F(X,XY)&=&\frac{(X^{q^n}+bX^{q^m})X^{q^t}Y^{q^t}-(X^{q^n}Y^{q^n}+bX^{q^m}Y^{q^m})X^{q^t}}{X^{q+1}(Y-Y^q)}\\
&=&X^{q^n+q^t-q-1}\frac{(1+bX^{q^m-q^n})Y^{q^t}-(Y^{q^n}+bX^{q^m-q^n}Y^{q^m})}{(Y-Y^q)}.
\end{eqnarray*}

We have that 
$G(X,Y)=F(X,XY)=0$ if and only if (apart from $X=0$)
$$bX^{q^m-q^n}\frac{Y^{q^t}-Y^{q^m}}{Y^q-Y}+\frac{Y^{q^t}-Y^{q^n}}{Y^q-Y},$$
that is 
$$bX^{q^m-q^n}=\frac{Y^{q^n}-Y^{q^t}}{Y^{q^t}-Y^{q^m}}.$$
Consider $U=X^{q^n}$, $V=Y^{q^{\min(t,n)}}$, therefore  
$$bU^{q^{m-n}-1}=\frac{V^{q^{n-\min(t,n)}}-V^{q^{t-\min(t,n)}}}{V^{q^{t-\min(t,n)}}-V^{q^{m-\min(t,n)}}}.$$

\begin{itemize}
    \item Suppose $t< n<m$. Then 
    $$bU^{q^{m-n}-1}=\frac{V^{q^{n-t}}-V}{V-V^{q^{m-t}}}.$$
    \item Suppose $n<t\leq m$. Then 
    $$bU^{q^{m-n}-1}=\frac{V-V^{q^{t-n}}}{V^{q^{t-n}}-V^{q^{m-n}}}=\frac{V-V^{q^{t-n}}}{(V-V^{q^{m-t}})^{q^{t-n}}}.$$
    \item Suppose $n<m<t$. Then 
     $$bU^{q^{m-n}-1}=\frac{V-V^{q^{t-n}}}{V^{q^{t-n}}-V^{q^{m-n}}}=\frac{V-V^{q^{t-n}}}{(V^{q^{t-m}}-V)^{q^{m-n}}}.$$
\end{itemize}

The above equations always define a Kummer extension of the rational function field with constant field $\mathbb{F}_{q^k}$ apart from the case $m-t=t-n$, that is the three integers are in arithmetical progression. Therefore, in all these cases, there are always pairs $(x_0,y_0)\in \mathbb{F}_{q^k}^2$ such that $F(x_0,y_0)=0$ and $y_0/x_0\notin \mathbb{F}_q$ and then $X^{q^n}+bX^{q^m}$ is not exceptional scattered.

\subsection{Particular trinomial in characteristic $2$}
Now we consider the following trinomial
$$f_k(X)=X^{2^{k-2}}+aX^{2^{k-1}}+bX^{2^k},$$
where $a,b\in \mathbb{F}_{2^n}^*$.
\begin{proposition}
The polynomial $f_k$, $k>2$, $a,b\in \mathbb{F}_{2^n}^*$, is not exceptional scattered of index $t=0$.
\end{proposition}
\proof
Consider the curve $\mathcal{C}_k$ associated with $f_k$.
$$\mathcal{C}_k : \frac{(X^{2^{k-2}}+aX^{2^{k-1}}+bX^{2^k})Y+(Y^{2^{k-2}}+aY^{2^{k-1}}+bY^{2^k})X}{XY(X+Y)}=0.$$
Let us consider the isomorphism $(X,Y)\mapsto (X,XY)$. The equation of the new curve is 
$$\frac{(X^{2^{k-2}}+aX^{2^{k-1}}+bX^{2^k})XY+(X^{2^{k-2}}Y^{2^{k-2}}+aX^{2^{k-1}}Y^{2^{k-1}}+bX^{2^k}Y^{2^k})X}{X^3Y(1+Y)}=0,$$
that is (dividing also by $X^{2^{k-2}-2}$)
$$\frac{Y^{2^{k-2}-1}+1}{Y+1}+aX^{2^{k-1}-2^{k-2}}\frac{Y^{2^{k-1}-1}+1}{Y+1}+bX^{2^k-2^{k-2}}\frac{Y^{2^k-1}+1}{Y+1}=0.$$
Let $U=X^{2^{k-2}}$, then the above equations reads
\begin{equation}\label{Eq:U}
bU^3+aU\frac{Y^{2^{k-1}-1}+1}{Y^{2^k-1}+1}+\frac{Y^{2^{k-2}-1}+1}{Y^{2^k-1}+1}=0,
\end{equation}
which defines an irreducible curve if and only if there is no solution $\overline{U}=\frac{F(Y)}{G(Y)}\in \overline{\mathbb{F}_q}(Y)$ of Equation \eqref{Eq:U}. If $k$ is even, then $\eta$ such that $\mathbb{F}_4^*=\langle \eta\rangle$ is a pole of multiplicity one of $\frac{Y^{2^{k-1}-1}+1}{Y^{2^k-1}+1}$ and it is not a pole nor a zero of $\frac{Y^{2^{k-2}-1}+1}{Y^{2^k-1}+1}$, a contradiction since the valuation of $U$ in the corresponding place of must be an integer. If $k$ is odd all the places corresponding to roots of $Y^{2^k-1}+1$ are not poles of $\overline{U}$ (same argument as above). All the other places are not poles of $\frac{Y^{2^{k-1}-1}+1}{Y^{2^k-1}+1}$  nor  of $\frac{Y^{2^{k-2}-1}+1}{Y^{2^k-1}+1}$ and therefore they are not poles of $\overline{U}$. This means that the unique pole of a solution $\overline{U}$ is $\infty$ and it has multiplicity $2^{k-2}$, that is $G(Y)$ is a constant and $F(Y)$ has degree $2^{k-2}$. This clearly gives a contradiction.  So  Equation \eqref{Eq:U} has no solution in $\overline{\mathbb{F}_q}(Y)$ and so it defines an absolutely irreducible $\mathbb{F}_{2^n}$-rational curve. The claim follows.
\endproof

\endproof


\begin{thebibliography}{99}

\bibitem{BZ}
D. Bartoli, Y. Zhou, Exceptional scattered polynomials, {\em J. Algebra} 509:507--534,  2018.

\bibitem{HKT} J.W.P. Hirschfeld, G. Korchm\'aros, and F. Torres, {\em Algebraic Curves over a Finite Field}, Princeton Series in Applied Mathematics, Princeton (2008).

\bibitem{Sti} H.  Stichtenoth, {\em Algebraic function fields and codes}, 2nd edn. Graduate Texts in Mathematics  254. Springer, Berlin (2009).

\bibitem{ball_linear_2000}
S.~Ball, A.~Blokhuis, and M.~Lavrauw, 
\newblock Linear $(q+1)$-fold blocking sets in $\mathrm{PG}(2, q^4)$,
\newblock {\em Finite Fields and Their Applications}, 6(4):294 -- 301, 2000.

\bibitem{BS2018} D. Bartoli, K.-U. Schmidt, Low-degree planar polynomials over finite fields of characteristic two, submitted.

\bibitem{bartoli_maximum_2017}
D.~Bartoli, M.~Giulietti, G.~Marino, and O.~Polverino,
\newblock Maximum scattered linear sets and complete caps in galois spaces,
\newblock {\em Combinatorica}, 38(2):255-278,  April 2018.

\bibitem{blokhuis_scattered_2000}
A.~Blokhuis and M.~Lavrauw,
\newblock Scattered spaces with respect to a spread in $\mathrm{PG}(n,q)$,
\newblock {\em Geometriae Dedicata}, 81(1):231--243, 2000.

\bibitem{blokhuis_two-intersection_2002}
A.~Blokhuis and M.~Lavrauw,
\newblock On two-intersection sets with respect to hyperplanes in projective
  spaces,
\newblock {\em Journal of Combinatorial Theory, Series A}, 99(2):377 -- 382,
  2002.

\bibitem{cardinali_semifield_2006}
I.~Cardinali, O.~Polverino, and R.~Trombetti,
\newblock Semifield planes of order $q^4$ with kernel {$F_{q^2}$} and center
  {$F_q$},
\newblock {\em European Journal of Combinatorics}, 27(6):940{\textendash}961,
  Aug. 2006.

\bibitem{caullery_large_class_2014}
F.~Caullery,
\newblock A new large class of functions not {APN} infinitely often,
\newblock {\em Designs, Codes and Cryptography}, 73(2):601--614, 2014.

\bibitem{caullery_classification_2015}
F.~Caullery and K.-U. Schmidt,
\newblock On the classification of hyperovals,
\newblock {\em Advances in Mathematics}, 283:195--203, Oct. 2015.

\bibitem{caullery_exceptional_2016}
F.~Caullery, K.-U. Schmidt, and Y.~Zhou,
\newblock Exceptional planar polynomials,
\newblock {\em Designs, Codes and Cryptography}, 78(3):605--613, 2016.

\bibitem{cossidente_non-linear_2016}
A.~Cossidente, G.~Marino, and F.~Pavese,
\newblock Non-linear maximum rank distance codes,
\newblock {\em Designs, Codes and Cryptography}, 79(3):597--609, June 2016.

\bibitem{csajbok_classes_2017}
B.~Csajb{\'o}k, G.~Marino, and O.~Polverino,
\newblock Classes and equivalence of linear sets in $\mathrm{PG}(1,q^n)$,
\newblock {\em Journal of Combinatorial Theory, Series A}, 157:402--426, July 2018.

\bibitem{csajbok_newMRD_2017}
B.~Csajb\'ok, G.~Marino, O.~Polverino, and C.~Zanella,
\newblock A new family of {MRD}-codes,
\newblock {\em Linear Algebra and its Applications}, 548:203--220, July 2018.


\bibitem{csajbok_maximum_2017}
B.~Csajb\'ok, G.~Marino, O.~Polverino, and F.~Zullo,
\newblock Maximum scattered linear sets and {MRD}-codes,
\newblock {\em Journal of Algebraic Combinatorics}, 46(3-4):517--531, December 2017.

\bibitem{csajbok_equivalence_2016}
B.~Csajb{\'o}k and C.~Zanella,
\newblock On the equivalence of linear sets,
\newblock {\em Designs, Codes and Cryptography}, 81(2):269--281, 2016.

\bibitem{delsarte_bilinear_1978}
P.~Delsarte,
\newblock Bilinear forms over a finite field, with applications to coding
  theory,
\newblock {\em Journal of Combinatorial Theory, Series A}, 25(3):226--241, Nov.
  1978.

\bibitem{durante_nonlinear_MRD_2016}
N.~Durante and A.~Siciliano,
\newblock Non-linear maximum rank distance codes in the cyclic model for the
  field reduction of finite geometries,
\newblock {\em Electronic Jornal of Combinatorics}, 24(2):P2.33, 2017.


\bibitem{durante_hyperovals_2017}
N.~Durante, R.~Trombetti, and Y.~Zhou,
\newblock Hyperovals in {K}nuth's binary semifield planes,
\newblock {\em European Journal of Combinatorics}, 62:77--91, May 2017.

\bibitem{ebert_infinite_2009}
G.~L. Ebert, G.~Marino, O.~Polverino, and R.~Trombetti,
\newblock Infinite families of new semifields.
\newblock {\em Combinatorica}, 29(6):637--663, Nov. 2009.

\bibitem{hernando_proof_2012}
H.~Fernando and G.~{McGuire},
\newblock Proof of a conjecture of {Segre} and {Bartocci} on monomial
  hyperovals in projective planes,
\newblock {\em Des. Codes Cryptography}, 65(3):275{\textendash}289, 2011.


\bibitem{gabidulin_MRD_1985}
E.~Gabidulin,
\newblock Theory of codes with maximum rank distance,
\newblock {\em Problems of information transmission}, 21:3--16, 1985.


\bibitem{hernando_proof_2011}
F.~Hernando and G.~{McGuire},
\newblock Proof of a conjecture on the sequence of exceptional numbers,
  classifying cyclic codes and {APN} functions.
\newblock {\em Journal of Algebra}, 343(1):78{\textendash}92, Oct. 2011.


\bibitem{horlemann-trautmann_new_2015}
A.-L. Horlemann-Trautmann and K.~Marshall,
\newblock New criteria for {MRD} and {Gabidulin} codes and some rank metric
  code constructions,
\newblock {\em Advances in Mathematics of Communications}, 11(3):533--548,  August  2017.

\bibitem{janwa_double-error-correcting_1995}
H.~Janwa, G.~Mcguire, and R.~Wilson,
\newblock Double-error-correcting cyclic codes and absolutely irreducible
  polynomials over {$GF(2)$},
\newblock {\em Journal of Algebra}, 178(2):665{\textendash}676, Dec. 1995.

\bibitem{jedlicka_apn_2007}
D.~Jedlicka.
\newblock {APN} monomials over {$GF(2^n)$} for infinitely many $n$, 
\newblock {\em Finite Fields and Their Applications},
  13(4):1006{\textendash}1028, Nov. 2007.

\bibitem{johnson_handbook_2007}
N.~L. Johnson, V.~Jha, and M.~Biliotti,
\newblock {\em Handbook of finite translation planes}, volume 289 of {\em Pure
  and Applied Mathematics {(Boca} Raton)},
\newblock Chapman \& {Hall/CRC}, Boca Raton, {FL}, 2007.

\bibitem{koetter_coding_2008}
R.~Koetter and F.~Kschischang, 
\newblock Coding for errors and erasure in random network coding,
\newblock {\em IEEE Transactions on Information Theory}, 54(8):3579--3591, Aug.
  2008.

\bibitem{kshevetskiy_new_2005}
A.~Kshevetskiy and E.~Gabidulin, 
\newblock The new construction of rank codes.
\newblock In {\em International {Symposium} on {Information} {Theory}, 2005.
  {ISIT} 2005. {Proceedings}}, pages 2105--2108, Sept. 2005.

\bibitem{lavrauw_scattered_2016}
M.~Lavrauw, 
\newblock Scattered spaces in galois geometry, \newblock in {\em Contemporary Developments in Finite Fields and Applications},
  pages 195--216. World Scientific, 2016.

\bibitem{lavrauw_field_reduction_2015}
M.~Lavrauw and G.~Van~de Voorde,
\newblock Field reduction and linear sets in finite geometry,
\newblock In G.~Kyureghyan, G.~Mullen, and A.~Pott, editors, {\em Contemporary
  Mathematics}, volume 632, pages 271--293. American Mathematical Society,
  2015.

\bibitem{leducq_functions_PN_2015}
E.~Leducq,
\newblock Functions which are {PN} on infinitely many extensions of
  $\mathbb{F}_p$, $p$ odd,
\newblock {\em Designs, Codes and Cryptography}, 75(2):281--299, 2015.

\bibitem{lunardon_linear_k-blocking_2001}
G.~Lunardon,
\newblock Linear k-blocking sets,
\newblock {\em Combinatorica}, 21(4):571--581, 2001.

\bibitem{lunardon_maximum_scattered_2014}
G.~Lunardon, G.~Marino, O.~Polverino, and R.~Trombetti,
\newblock Maximum scattered linear sets of pseudoregulus type and the {Segre}
  variety $\mathcal{S}_{n,n}$,
\newblock {\em Journal of Algebraic Combinatorics}, 39(4):807--831, 2014.

\bibitem{lunardon_blocking_2000}
G.~Lunardon and O.~Polverino,
\newblock Blocking sets of size $q^t+q^{t-1}+1$,
\newblock {\em Journal of Combinatorial Theory, Series A}, 90(1):148--158, Apr.
  2000.

\bibitem{lunardon_blocking_2001}
G.~Lunardon and O.~Polverino,
\newblock Blocking sets and derivable partial spreads,
\newblock {\em Journal of Algebraic Combinatorics}, 14(1):49--56, July 2001.

\bibitem{lunardon_generalized_2015}
G.~Lunardon, R.~Trombetti, and Y.~Zhou,
\newblock Generalized twisted {Gabidulin} codes,
\newblock {\em Journal of Combinatorial Theory, Series A}, 159:79--106, October 2018.

\bibitem{lunardon_kernels_2016}
G.~Lunardon, R.~Trombetti, and Y.~Zhou,
\newblock On kernels and nuclei of rank metric codes.
\newblock {\em Journal of Algebraic Combinatorics}, 46(2):313--340, September 2017.

\bibitem{marino_towards_2011}
G.~Marino, O.~Polverino, and R.~Trombetti,
\newblock Towards the classification of rank 2 semifields 6-dimensional over
  their center,
\newblock {\em Designs, Codes and Cryptography. An International Journal},
  61(1):11--29, 2011.

\bibitem{morrison_equivalence_2014}
K.~Morrison,
\newblock Equivalence for rank-metric and matrix codes and automorphism groups
  of gabidulin codes,
\newblock {\em {IEEE} Transactions on Information Theory}, 60(11):7035--7046,
  2014.

\bibitem{neri_genericity_2017}
A.~Neri, A.-L. Horlemann-Trautmann, T.~Randrianarisoa, and J.~Rosenthal,
\newblock On the genericity of maximum rank distance and gabidulin codes,
\newblock {\em Designs, Codes and Cryptography}, pages 1--23, 2017.
\newblock Online First.

\bibitem{otal_additive_2016}
K.~Otal and F.~\"Ozbudak,
\newblock Additive rank metric codes,
\newblock {\em IEEE Transactions on Information Theory}, 63(1):164--168, Jan
  2017.

\bibitem{polverino_linear_2010}
O.~Polverino,
\newblock Linear sets in finite projective spaces,
\newblock {\em Discrete Mathematics}, (22):3096--3107, 2010.

\bibitem{rodier_functions_APN_2011}
F.~Rodier,
\newblock Functions of degree $4e$ that are not {APN} infinitely often,
\newblock {\em Cryptography and Communications}, 3(4):227--240, 2011.

\bibitem{schmidt_planar_2014}
K.-U. Schmidt and Y.~Zhou,
\newblock Planar functions over fields of characteristic two,
\newblock {\em Journal of Algebraic Combinatorics}, 40(2):503--526, Sept. 2014.

\bibitem{schmidt_number_MRD_2017}
K.-U. Schmidt and Y.~Zhou,
\newblock On the number of inequivalent {MRD} codes,
\newblock {\em Designs, Codes and Cryptography}, 86(9):1973--1982, September 2018.

\bibitem{sheekey_new_2016}
J.~Sheekey,
\newblock A new family of linear maximum rank distance codes,
\newblock {\em Advances in Mathematics of Communications}, 10(3):475--488,
  2016.

\bibitem{zieve_planar_2015}
M.~E. Zieve,
\newblock Planar functions and perfect nonlinear monomials over finite fields,
\newblock {\em Designs, Codes and Cryptography}, 75(1):71--80, 2015.


\end{thebibliography}
\end{document}